\documentclass[12pt]{article}
\usepackage{amsmath}
\usepackage{amssymb}
\usepackage{latexsym}
\usepackage{enumerate}
\usepackage[ansinew]{inputenc}
\usepackage[all]{xy}
\usepackage[usenames]{color}
\usepackage{epsfig,graphicx,epstopdf}

\newtheorem{theorem}{Theorem}[section]
\newtheorem{proposition}[theorem]{Proposition}
\newtheorem{corollary}[theorem]{Corollary}
\newtheorem{lemma}[theorem]{Lemma}
\newtheorem{example}[theorem]{Example}
\newtheorem{remark}[theorem]{Remark}

\newtheorem{definition}[theorem]{Definition}
\newenvironment{proof}{\noindent{\sc Proof.}}{\quad\qed\medskip}

\newcommand{\F}{{\cal F}}
\newcommand{\E}{{\underline{\underline{E}}}}

\newcommand{\gd}{{\underline{\underline{\text{\rm gd}}}\ }}
\newcommand{\gde}{{\underline{\underline{\emph{gd}}}\ }}
\newcommand{\gdp}{{{\underline{\text{\rm gd}}}\ }}

\newcommand{\qed}{\quad\lower0.05cm\hbox{$\Box$}}

\newcommand{\downarrowright}[1]{\downarrow
\rlap{\raise0.1cm\hbox{$\scriptstyle{#1}$}}}

\newcommand{\downarrowleft}[1]{\rlap{\kern-0.2cm
\raise0.1cm\hbox{$\scriptstyle{#1}$}}\downarrow}

\newcommand{\uparrowright}[1]{\uparrow
\rlap{\lower0.1cm\hbox{$\scriptstyle{#1}$}}}

\newcommand{\uparrowleft}[1]{\rlap{\kern-0.2cm
\lower0.1cm\hbox{$\scriptstyle{#1}$}}\uparrow}

\newcommand{\CommH}{\textrm{Comm}_G(H)}
\newcommand{\Comm}{\textrm{Comm}}

\definecolor{byzantium}{rgb}{0.44, 0.16, 0.39}

\begin{document}
\setcounter{page}{1}
\title{Classifying spaces for the family of virtually cyclic subgroups of braid groups}

\author{Ramón Flores and Juan Gonz\'alez-Meneses \footnote{The first author was supported by MEC-FEDER grant MTM2013-42293-P and the second author was supported by MEC-FEDER grant MTM2013-44233-P. Both authors supported by MEC-FEDER grant MTM2016-76453-C2-1-P}}
\date{February 7th, 2018}
\maketitle

\begin{abstract}

We prove that, for $n\geq 3$, the minimal dimension of a model of the classifying space of the braid group $B_n$, and of the pure braid group $P_n$, with respect to the family of virtually cyclic groups is $n$.

\end{abstract}


\section{Introduction}

In the last years, there has been a growing interest in the description of the classifying space of a group with respect to the family of its virtually cyclic subgroups, usually denoted $\underline{\underline{E}}G$. The main reason for this interest is that this space is the target of the Farrell-Jones conjecture \cite{FaJo93}, which given a group $G$, intends to compute the (algebraic) $K$-theory of the group ring $\mathbb{Z}G$ through an assembly map that uses as a source non-connective (topological) $K$-theory of the classifying space. The conjecture has been proved to be true for a number of groups, see \cite{LuRe05} for an excellent survey.

The knowledge of the algebraic $K$-theory of $\mathbb{Z}G$ is based in part on finding manageable and ``small" models for $\underline{\underline{E}}G$, and in particular in understanding which is the minimal possible dimension of such a model. Hence, this problem has been widely studied in the last years, and computations are available for an important number of classes of groups, as for example locally finite groups \cite{DePe13}, polycyclic groups \cite{LuWe12} and more generally elementary amenable groups \cite{DePe14}, \cite{FlNu13}, CAT(0)-groups \cite{Luc09}, \cite{DePe15}, linear groups \cite{DKP15}, hyperbolic groups \cite{JuLe06}, \cite{LaOr07} or mapping class groups \cite{DePe15}, \cite{JuTr16}.

We are interested in the (full) braid groups $B_n$, which in fact can be seen as a particular instance of mapping class groups. These objects, aside from being interesting by their own right, have appeared in very different fields of Mathematics, as for example algebra, topology, physics or geometry. In our context, it has been proved by Juan-S\'anchez \cite{JuSa16} that Farrell-Jones isomorphism holds for them, and moreover their lower algebraic $K$-theory groups have been computed \cite{GuJu15}. However, not a lot is known about concrete models for $\underline{\underline{E}}B_n$, not even the minimal dimension of these spaces. In his PhD Thesis \cite{Flu12}, M. Fluch (Proposition 4.16), estimated this dimension between 3 and 5 in the case of $B_3$. More recently, Juan-Trujillo \cite{JuTr16} bounded it for a general $B_n$ with a factorial bound (more information at the end of Section 3).

The main result of our paper (Theorem \ref{teoremazo}) is that the minimal dimension of a model for $\underline{\underline{E}}B_n$ (and also of the pure version $\underline{\underline{E}}P_n$), for every $n\geq 3$, is $n$. In this way, it is proved that the equality in Question 1.2 of \cite{Luc09} holds for these groups, and we expect that our result is helpful for computations of Bredon (co)homology and/or algebraic $K$-theory. Of course, next challenge in this context will be to find explicit models of $\underline{\underline{E}}B_n$ that realize the minimal dimension, and in this sense, a promising line of research concerns the actions of braid groups on CAT(0)-spaces (see \cite{HKS16} and references there). On the other hand, our methods strongly depend on the rich internal structure of the (full) braid groups, so it is likely they can be applied to braid groups over other surfaces for which the virtually cyclic subgroups are understood (see for example \cite{GoGu13} for the case of the sphere).

The structure of the paper is the following. In Section 2 we provide the background about classifying spaces and geometric dimensions that will be needed later, and in particular, we describe L\"{u}ck-Weiermann model for the case of virtually cyclic subgroups. Section 3 is devoted to braid groups, and in it we recall general facts about these objects, and prove some new results that will be useful for our computations. In Section 4 we describe the structure of the commensurators of the cyclic subgroups. Section 5 contains the proof of our main theorem. The most difficult point of the proof is to bound the dimensions in the reducible non-periodic case, and this goal is achieved by proving that the quotient of certain normalizers in $B_n$ are virtually torsion-free, and then applying an appropriate result of C. Mart\'{i}nez (Theorem 2.5 in \cite{Mar13}) that only works in this case.

\section{Background}

\subsection{Classifying space for families}

In this section we review the notion of classifying space for a family of subgroups of a group, which will be the main object of study in this note, with special emphasis in the case of the family of virtually cyclic subgroups. We intend to give a brief survey based on the treatment of \cite{DePe14}, the reader interested in a more thorough approach is referred to L\"{u}ck excellent monography \cite{Luc05}. Moreover, we assume that the reader is familiar with the notion of $G$-CW-complex and other fundamental concepts of equivariant $G$-homotopy; in any case, all the necessary information can be found in \cite{MiVa03}, Part I, section 2.

\begin{definition}

Let $G$ be a discrete group, and $\mathcal{F}$ a family of subgroups of $G$ which is closed under conjugation and passing to subgroups. Then a $G$-CW-complex $X$ is a \emph{classifying space for the family} $\mathcal{F}$ if the fixed point set $X^H$ is contractible for every $H\in \mathcal{F}$ and empty for every $H\notin \mathcal{F}$.

\end{definition}

The classifying space for the family $\mathcal{F}$ is always unique up to $G$-homotopy equivalence, and is denoted by $E_{\mathcal{F}}G$. Observe that the definition implies that the isotropy groups of the action of $G$ over any model for $E_{\mathcal{F}}G$ are in $\mathcal{F}$. Moreover, as the family is closed under passing to subgroups, the trivial group is always a member of $\mathcal{F}$, and this implies that $E_{\mathcal{F}}G$ is always a contractible $G$-space. Remark also that if $G\in \mathcal{F}$, the point is a model for $E_{\mathcal{F}}G$.

\begin{remark}
\label{subfamilies}
Given a family $\F$ of subgroups of $G$ closed under conjugation and subgroups and a subgroup $H<G$, the family $\F\cap H$ whose elements are the intersections of $H$ with the elements of $\F$ is a family of subgroups of $H$ for which the same holds. Actually, $\F\cap H$ is precisely the set of subgroups of $H$ which belong to $\F$. If we consider the action of $H$ over $E_{\F}G$ by restriction, $E_{\F}G$ becomes a model also for $E_{\F\cap H}H$.
\end{remark}

\begin{example}

The following are the most important instances of classifying spaces for families. Each one deserves its own notation:

\begin{itemize}

\item If $\mathcal F_{\{1\}}$ is the family whose only element is the trivial subgroup, $E_{\mathcal{F}_{\{1\}}}G=EG$, the universal space for principal $G$-bundles.

\item If $\mathcal{F}_{Fin}$ is the family of finite subgroups of $G$, then $E_{\mathcal{F}_{Fin}}G$ is usually denoted $\underline{E}G$ and called the \emph{classifying space for proper actions} of $G$.

 \item If $\mathcal{F}_{vc}$ is the family of virtually cyclic subgroups of $G$, the space $E_{\mathcal{F}_{vc}}G$ is denoted $\E G$.

\end{itemize}

\end{example}

Observe that in a torsion-free group $G$, $\F_{\{1\}}=\F_{Fin}(G)$. In particular, in this case, $EG=\underline{E}G$. Also, it is well-known that in a torsion-free group $G$, every virtually cyclic subgroup is cyclic (\cite{Mac96}, Lemma 3.2), so $\mathcal F_{vc}(G)$ is the set of cyclic subgroups of $G$.

The classifying space for proper actions has been widely studied from the nineties, as it is the target of the Baum-Connes conjecture \cite{BCH94}, which remains  unsolved in full generality. However, although the classifying space $\E G$ is the object of study of the Farrell-Jones conjecture \cite{FaJo93} about the algebraic $K$-theory and $L$-theory of the group ring of $G$, also unsolved in general, much less is known about this classifying space, because a construction of a manageable model for it is usually much more difficult than the construction of a good model for $\underline{E}G$. In this sense, a big step forward has been the work of L\"{u}ck-Weiermann, which builds a model for $\E G$ using as building pieces models $E_{\mathcal{F}}H$, with $H\leq G$ and $\mathcal{F}$ families of groups which are contained in the family $\mathcal{F}_{vc}(G)$. Because of its importance in our development, we will describe now this model; details can be found in \cite{LuWe12}.

Let $G$ be a group. An equivalence relation in the set $\F_{vc}\backslash \F_{Fin}$ of infinite virtually cyclic subgroups of $G$ is defined by establishing that $H\sim K$ if and only if $H\cap K$ is infinite. We will denote by $[H]$ the equivalence class of a subgroup $H$ by this relation. Notice that conjugation in $G$ preserves $\F_{vc}$, $\F_{Fin}$ and the relation $\sim$, hence we can define $g^{-1}[H]g$ as $[g^{-1}Hg]$ for any $g\in G$ and any $H$ in $\F_{vc}\backslash \F_{Fin}$.

\begin{definition}
\label{eqclass}
If $[H]$ is an equivalence class by $\sim$, the \emph{normalizer} of $[H]$ is defined as the subgroup
$$N_G[H]:=\{g\in G\:|\ g^{-1}[H]g=[H]\}.$$
\end{definition}

It is important to observe that $N_G[H]$ contains the normalizer $N_G(H)$ for every representative $H$ of the class $[H]$, but in general is not equal to any of these normalizers (Example 2.6 in \cite{LuWe12}). However, in the groups we will be working with, there always will be a group $H_0\in [H]$ such that $N_G[H]=N_G(H_0)$.  In general, we will see in Section~\ref{Commensurators} that for every group $G$, one has $N_G[H]=\text{Comm}_G(H)$, the commensurator of $H$ in $G$ (see Definition \ref{Commdefo}).

Now a family of subgroups of $N_G[H]$ can be defined as follows:
$$\F [H]:=\{K<N_G[H] \: |\ K\sim H \text{ or } |K|<\infty \}.$$
Observe that this family is closed under conjugation (in $N_G[H]$) and taking subgroups.

Using these subgroups and families, an appropriate model for $\E G$ can be defined:

\begin{theorem}
\label{maintheorem}{\rm (\cite{LuWe12}, Theorem 2.3)} Let $I$ be a complete set of representatives of the $G$-orbits (under conjugation) of equivalence classes $[H]$ of infinite virtually cyclic subgroups of $G$. For every $[H]\in I$, choose models for $\underline{E} N_G[H]$ and $E_{\F [H]}N_G[H]$. Choose also a model for $\underline{E}G$. Now consider the $G$-pushout:
$$
\xymatrix{  \coprod_{[H]\in I}G\times_{N_G[H]}\underline{E}N_G[H] \ar[r]^{\hspace{2cm} i} \ar[d]^{\coprod_{[H]\in I}id_G\times_{N_G[H]}f_{[H]}}  & \underline{E}G \ar[d] \\
\coprod_{[H]\in I}G\times_{N_G[H]}E_{\F [H]}N_G[H]  \ar[r] & X }
$$
where $f_{[H]}$ is a cellular $N_G[H]$-map for every $[H]\in I$ and $i$ is the inclusion, or $i$ is a cellular $G$-map and $f_{[H]}$ is an inclusion for every $[H]\in I$. Then $X$ is a model for $\E G$.
\end{theorem}

In practice, this theorem implies that the existence of good models for the proper classifying space of the commensurators and $G$, and also of the classifying spaces with respect to the families $\F [H]$ will lead to the knowledge of good models for $\E G$. Moreover, the push-out implies the existence of a long exact sequence in Bredon homology, and dimensional consequences that we will analyze in next section.

Although we have described here only the particular case of finite and virtually cyclic subgroups (as it is the only one that will be needed), it is worth to point out that L\"{u}ck-Weiermann model is defined for any two families $\cal{F}$ and $\cal{G}$ such that $\cal{F}\subseteq \cal{G}$ and an equivalence relation subject to some conditions is defined in $\cal{G}\backslash \cal{F}$ (see \cite{LuWe12}, Section 2).

\subsection{Minimal dimensions for $\E G$}

In this section we will recall the results about geometric dimensions of classifying spaces that will be necessary in the remaining of this note.

\begin{definition}

Let $G$ a group, $\F$ a family closed under conjugation and taking subgroups. The \emph{geometric dimension} of $G$ with respect to the family $\F$ is the minimal dimension $\emph{gd}_{\F} G$ of a model for $E_{\F}G$.

\end{definition}

According to the classical notation, we denote $\textrm{gd}_{\F}G$ by $\textrm{gd }G$ when $\F=\F_{\{1\}}$; by $\underline{\textrm{gd}}\textrm{ }G$ when $\F=\F_{Fin}$; and by $\gd G$ when $\F=\mathcal{F}_{vc}$. We are mainly interested in the latter.

The study of the dimension of a group is a classic topic of research where group theory, homological algebra and geometry overlap via the different geometric and (co)homological versions of the dimension. Usually, these different versions are related, and small values of them allow a sharp knowledge of the structure of the group. Of course, in this paper we are mainly interested in the geometric dimension with respect to the family of virtually cyclic groups, but in our arguments we will also need to compute some proper geometric dimensions, as well as some algebraic counterparts of it. Among them, it will be particularly important the next one:

\begin{definition}

Given a group $G$ and a family $\mathcal{F}$ of subgroups of $G$ closed under conjugation and taking subgroups, the (Bredon) cohomological dimension of $G$ with respect to the family $\mathcal{F}$, denoted by $\emph{cd}_{\mathcal{F}}G,$ is the maximal dimension of a nonzero Bredon cohomology group with respect to $\mathcal{F}$.

\end{definition}

The main definitions about Bredon (co)homology can be found in \cite{MiVa03} or \cite{Flu12}. In particular, the following fact is described in the second reference, and it will be relevant for our purposes:

\begin{proposition}
\label{Eilenberg}

In the notation of the previous definition, for every $G$ and $\mathcal{F}$ the equality $\rm{cd}_{\mathcal{F}}G=\rm{gd}_{\mathcal{F}}G$ holds whenever $\rm{gd}_{\mathcal{F}}G\neq 3$. If $\rm{gd}_{\mathcal{F}}G=3$, then $2\leq \rm{cd }_{\mathcal{F}}G\leq 3$.

\end{proposition}

The question about the equality when $\rm{gd }_{\mathcal{F}}G=3$ is known as Eilenberg-Ganea conjecture, and it is yet unsolved in the case of the trivial family. The paper \cite{BLN01} contains an interesting discussion about this conjecture, including examples of the strict inequality for the family of finite groups; examples for the family of virtually cyclic groups can be found in \cite{FlLe14}.

The key result that will lead us to identify $\gd B_n$ is a bound implied by the previously described model of L\"{u}ck-Weiermann:

\begin{corollary}\label{C:cotas_suficientes} (Remark 2.5 in \cite{LuWe12}) With the above notations, suppose there exists a natural number $k$ such that:
\begin{itemize}
\item $\gdp G\leq k$.
\item $\gdp N_G[H]\leq k-1$, for every $[H]\in I$.
\item $\text{\rm gd}_{\F[H]}N_G[H]\leq k$, for every $[H]\in I$.

\end{itemize}
Then $\gd G\leq k$.

\end{corollary}

When applying this corollary, it is interesting to realize that according to Remark \ref{subfamilies}, for every $H<G$ and every family of subgroups of $G$ closed under conjugation and subgroups, $\textrm{gd}_{\F\cap H} H\leq \textrm{gd}_{\F} G$. Then, $\gdp G$ and $\gd G$ work as upper bounds for the corresponding minimal dimension of any of their subgroups.






We should also mention here Proposition 5.1 in \cite{LuWe12}, which establishes that an upper bound for $\gd$ automatically gives a bound for $\underline{\textrm{gd}}$:

\begin{proposition}

For any discrete group G, $\underline{\emph{gd}}\textrm{ }G\leq \gde G+1$.

\end{proposition}

Note that the equality is possible: for example, as $\mathbb{Z}$ is virtually cyclic, $\gd \mathbb{Z}=0$, while $\gdp\mathbb{Z}=1$.  Usually, however, $\gdp G\leq \gd G$, and it is a question of L\"{u}ck \cite{Luc05} to identify families for which the inequality $\gd G\leq \gdp G+1$ holds, and also for which this inequality is an equality. For example, Degrijse-Petrosyan (Corollary 4.4 and Example 6.5 in \cite{DePe14}), prove that the inequality holds for elementary amenable groups, and also construct Bestvina-Brady-like counterexamples for which it does not hold. While it is easy to check that the inequality holds for braid groups (see Proposition \ref{P:abelian_inside} and the previous discussion), to prove that it is in fact an equality has been one of the main motivations of our paper, and is a beautiful and immediate consequence of Theorem \ref{teoremazo}.



\section{Braid groups}
\label{S:braids}

\subsection{A little survey}

In this section we collect the basic definitions and facts about braid groups. For further information the interested reader is referred to the more detailed treatments of \cite{GuJu15} or \cite{Gon11}

Given an integer $n\geq 2$, the braid group $B_n$ on $n$ strands can be defined in several different ways. Algebraically, it is given by the following presentation~\cite{Art25,Art47}:
$$
B_n=\left\langle \sigma_1,\ldots, \sigma_{n-1} \left| \begin{array}{cl} \sigma_i\sigma_j=\sigma_j\sigma_i, & \text{ if } |i-j|>1 \\ \sigma_i\sigma_j\sigma_i=\sigma_j\sigma_i\sigma_j, & \text{ if } |i-j|=1 \end{array}\right.\right\rangle
$$
Also, if $M_n$ is the configuration space of $n$ distinct points in the plane $\mathbb C$,
$$
   M_n =\{(x_1,\ldots,x_n)\in \mathbb C^n\: |\ x_i\neq x_j \text{ for } i\neq j\},
$$
and $\Sigma_n$ is the symmetric group on $n$ elements, which acts on $M_n$ by permuting coordinates, then
$$
    B_n=\pi_1(M_n/\Sigma_n).
$$
This can be visualized by choosing $n$ base points, say $\{1,\ldots,n\}\subset \mathbb C$, and considering a braid as a collection of $n$ disjoint paths in $\mathbb C \times [0,1]$, called strands, where the $i$-th strand starts at $(i,0)$, moves monotonically on the second coordinate, and ends at $(j,1)$ for some $j\in \{1,\ldots,n\}$. This collection of strands (i.e. this braid) is considered up to isotopy of $\mathbb C\times [0,1]$, fixing the boundary pointwise, and the multiplication of braids is given by stacking and rescaling. The generator $\sigma_j$ corresponds to a braid in which the strands $j$ and $j+1$ cross, while the other $n-2$ strands are constant. More precisely, $\sigma_j$ can be given by the paths
$$
\left(1,t\right),\ldots, \left(j-1, t\right), \;\left(j+\frac{1}{2}+\frac{1}{2}e^{(1-t)i\pi},t \right),\; \left(j+\frac{1}{2}+\frac{1}{2}e^{-ti\pi},t\right),\; \left(j+2,t\right),\ldots,\left(n,t\right)
$$
for $t\in [0,1]$.

Each braid has a corresponding permutation, induced by the endpoints of its strands, so there is a surjective map $B_n\rightarrow \Sigma_n$. The kernel of this map, called the {\em pure braid group} and denoted $P_n$, is the subgroup of braids in which the $i$-th strand starts at $(i,0)$ and ends at $(i,1)$ for $i=1,\ldots,n$. Actually, $P_n=\pi_1(M_n)$.~\cite{Bir74}

There is a third definition of $B_n$ that will be important for us in this paper. Let $\mathbb D_n$ be a $n$-times punctured disc (for instance, $\mathbb D_n=\mathbb D\backslash \{1,\ldots,n\}$ where $\mathbb D$ is the disc in $\mathbb C$ with diameter $[0,n+1]\subset \mathbb R$). Then a braid can be seen as an orientable homeomorphism of $\mathbb D_n$ to itself, fixing $\partial(\mathbb D_n)$ pointwise, up to isotopy relative to $\partial(\mathbb D_n)$. That is, $B_n$ is the {\em mapping class group} of $\mathbb D_n$~\cite{Hur91,Bir74}. We remark that an automorphism corresponding to a braid can be obtained from viewing the family of strands as a motion of the $n$ punctures in $\mathbb D$, which can be extended to a continuous motion of all points of $\mathbb D$, fixing the boundary. The final position of the points yields the homeomorphism, which is unique up to isotopy of $\mathbb D_n$ relative to $\partial(\mathbb D_n)$.  Conversely, if an automorphism $f$ of $\mathbb D_n$ is given which fixes the boundary $\partial(\mathbb D_n)$ pointwise, it can be uniquely extended to an automorphism $\overline f$ of the whole disk $\mathbb D$. As every automorphism of $\mathbb D$ fixing the boundary is isotopic to the trivial automorphism (Alexander's trick), we can continuously deform $\mbox{id}_\mathbb D$ into $\overline f$. The trace of the points $1,\ldots,n$ under this deformation yields a family of $n$ paths in $\mathbb D\times [0,1]$, which can be extended to $\mathbb C\times [0,1]$ by the identity outside the cylinder. These $n$ paths form the braid corresponding to the isotopy class of $f$.

The braid group $B_n$ is torsion-free~\cite{Chow48}. Its center is infinite cyclic, generated by
$$
\Delta^2= (\sigma_1(\sigma_2\sigma_1)\cdots (\sigma_{n-1}\sigma_{n-2}\cdots \sigma_1))^2
$$
This element corresponds to a Dehn-twist along a curve parallel to the the boundary of $\mathbb D_n$ (roughly speaking, a rotation of $\partial(\mathbb D_n)$ by $360$ degrees). Hence, to quotient $B_n$ by its center corresponds to collapsing the boundary to a new puncture, so $B_n/\langle \Delta^2\rangle$ is a subgroup (of index $n+1$) of the mapping class group of the $(n+1)$-times punctured sphere $\mathbb S_{n+1}$.

This way of viewing braids as mapping classes allows to use the powerful theory of Nielsen-Thurston~\cite{Thu88}, and to classify braids into three geometric types. In this way, a braid $\alpha$ is said to be:
\begin{itemize}

 \item {\em Periodic}, if some nontrivial power of $\alpha$ belongs to $\langle \Delta^2\rangle$.

 \item {\em Pseudo-Anosov}, if there is a pair of transverse measured foliations of $\mathbb D_n$ preserved by $\alpha$, such that the action of a homeomorphism $f$ (representing $\alpha$) on one of them scales the measure by some real number $\lambda>1$ (called {\em dilatation factor}), and the measure of the other one by $\lambda^{-1}$.

 \item {\em Reducible non-periodic}, if $\alpha$ preserves a family of isotopy classes of disjoint, essential simple closed curves in $\mathbb D_n$. Here {\em essential} means enclosing more than one and less than $n$ punctures.

\end{itemize}

Essentially, reducible braids are those which can be {\em reduced} into simpler ones. More precisely, suppose that $\alpha$ is a reducible, non-periodic braid. Up to conjugacy in $B_n$, we can assume that the family of isotopy classes of curves preserved by $\alpha$ can be represented by a family of circles. Up to replacing $\alpha$ by a power, we can assume that $\alpha$ is represented by an automorphism $f$ which sends every circle to itself, so we can restrict the automorphism $f$ to the connected components obtained by removing the family of circles from $\mathbb D_n$. As each of these connected components is again homeomorphic to a punctured disc, each restriction corresponds to a braid with fewer strands. Actually, there is a particular family of isotopy classes of curves, called the {\em canonical reduction system} of $\alpha$, $CRS(\alpha)$, such that each of the mentioned restrictions is either periodic or pseudo-Anosov (see \cite{BLM83}).

The decomposition of a reducible braid into simpler ones will be used several times in this paper, so we will make it more precise. Given a reducible, non-periodic braid $\alpha$, let $\mathcal C_\alpha$ be the set of outermost isotopy classes of curves in the canonical reduction system $CRS(\alpha)$. Notice that $\alpha$ preserves $\mathcal C_\alpha$ (up to isotopy). Up to conjugacy, we can assume that $\mathcal C_\alpha$ can be represented by a disjoint union of unnested circles. Now recall that an automorphism representing $\alpha$ can be extended to an automorphism $\overline f$ of $\mathbb D$, and that there is an isotopy from $\mbox{id}_{\mathbb D}$ to $\overline f$. The trace of the punctures under this isotopy represent the strands of $\alpha$, and the trace of the circles representing $\mathcal C_\alpha$ looks like a family of {\it tubes}, each one enclosing more than one and less than $n$ punctures. See Figure~\ref{F:External_braid}.

\begin{figure}[ht]
\begin{center}
  \includegraphics{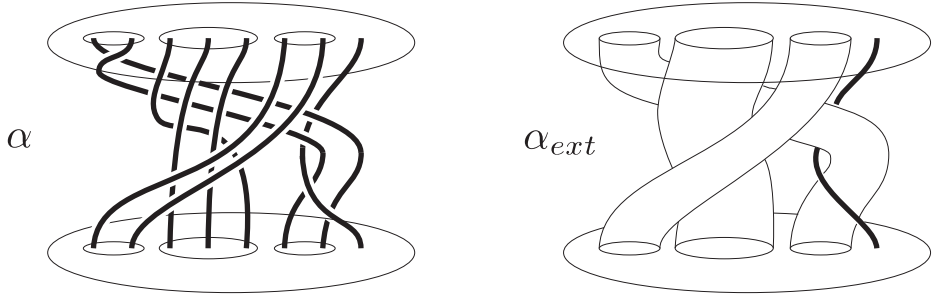}
\end{center}
\caption{A reducible braid $\alpha\in B_8$, and the external braid $\alpha_{ext}$. In this case, $\mathcal C_{\alpha}$ is represented by a family of three circles. The internal braids inside each tube are, respectively, $\sigma_1\in B_2$, $\sigma_1\sigma_2\in B_3$ and the trivial braid in $B_2$. The external braid is $\sigma_1 \sigma_2 \sigma_3^{-1} \sigma_1 \sigma_3^{-1}\in B_4$.}
\label{F:External_braid}
\end{figure}

The decomposition of $\alpha$ along $\mathcal C_\alpha$ yields an {\it external braid} $\alpha_{ext}$, and some {\it internal braids}. The external braid $\alpha_{ext}$ corresponds to forgetting what happens inside the tubes: one gets a braid made of tubes and (possibly some) strands, which is transformed into a braid by shrinking each tube to form a single strand. The internal braids are those which are contained into each tube. Actually, as a tube does not necessarily end at the same place it started, there is some shifting of the final points of the strands that should be made in order to define the internal braids properly, but it is clear from the picture (see Figure~\ref{F:External_braid}). See~\cite{Gon11b} for details.

A further simplification can still be done. Denote $C_1,\ldots, C_m$ the circles representing $\mathcal C_\alpha$. The action of $\alpha$ on $\mathcal C_{\alpha}$ induces a permutation on these circles. If $C_{i_1},\ldots,C_{i_r}$ is a cycle under this permutation, then we can conjugate $\alpha$ in such a way that all internal braids in the tubes corresponding to $C_{i_2},\ldots,C_{i_r}$ are trivial, and the only possibly nontrivial one is inside the tube corresponding to $C_{i_1}$. We can do the same with all cycles, so we can assume (up to conjugating $\alpha$) that there is just one internal braid $\alpha_i$ for each cycle of tubes determined by $\alpha$. These internal braids $\alpha_1,\ldots,\alpha_k$ are unique up to conjugacy (see \cite{GoWi04}), and can be used to describe the centralizer of $\alpha$:

\begin{theorem}\label{T:centralizer_Bn} {\rm (Theorem 1.1 in \cite{GoWi04})} Let $\alpha\in B_n$. If $n=2$ the centralizer $Z(\alpha)=B_2\simeq \mathbb Z$. If $n\geq 3$, the centralizer $Z(\alpha)$ is as follows:
\begin{itemize}

  \item If $\alpha$ is periodic, $Z(\alpha)$ either is equal to $B_n$ or is isomorphic to a subgroup of $B_m$ for some $m<n$. (Actually, in the latter case it is the braid group of an annulus on less than $n-1$ strands, which can be embedded into $B_m$).

  \item If $\alpha$ is pseudo-Anosov, $Z(\alpha)\simeq \mathbb Z^2$, consisting of all elements having a representative which preserves the same two foliations as $\alpha$.

  \item If $\alpha$ is reducible, there is a split short exact sequence
  $$
     1\longrightarrow Z(\alpha_1)\times \cdots \times Z(\alpha_k)\longrightarrow Z(\alpha) \stackrel{p_\alpha}{\longrightarrow} Z_0(\alpha_{ext})\rightarrow 1,
  $$
  where $\alpha_1,\ldots,\alpha_k$ and $\alpha_{ext}$ are defined as above, and $Z_0(\alpha_{ext})$ is a finite index subgroup of $Z(\alpha_{ext})$.
\end{itemize}
\end{theorem}

The map $p_\alpha:\: Z(\alpha) \rightarrow Z(\alpha_{ext})$ mentioned above goes as follows: given $\beta\in Z(\alpha)$, one has $\beta^{-1}\alpha\beta=\alpha$. The canonical reduction system of a braid behaves in a natural way with respect to conjugations~\cite{BLM83}, meaning that $CRS(\beta^{-1}\alpha \beta)$ is the image of $CRS(\alpha)$ under the mapping class $\beta$. But $CRS(\beta^{-1}\alpha \beta)=CRS(\alpha)$, hence $\beta$ sends $CRS(\alpha)$ to itself, and then it sends $\mathcal C_{\alpha}$ to itself, too. Therefore, $\beta$ can be decomposed along $\mathcal C_{\alpha}$, and $p_\alpha(\beta)$ is just the external braid corresponding to this decomposition. Notice that we do not necessarily have $\mathcal C_\beta =\mathcal C_\alpha$ (actually, $\beta$ could even be periodic and $\mathcal C_\beta$ could be empty), so $p_\alpha(\beta)$ is not necessarily $\beta_{ext}$ (which could be not defined). $p_\alpha(\beta)$ is the {\it tubular braid} described by the action of $\beta$ on $\mathbb D_n$, by looking at the motion of the circles representing $\mathcal C_\alpha$ and of the punctures not enclosed by $\mathcal C_\alpha$.

\subsection{Centralizers and roots}

In this section we will state some further properties of centralizers and roots in the pure braid group $P_n$, that will be useful later.

\begin{proposition}\label{P:unique_roots}
Roots in $P_n$ are unique. That is, if $\alpha, \beta \in P_n$ are such that $\alpha^m=\beta^m$ for some $m\neq 0$, then $\alpha=\beta$.
\end{proposition}

\begin{proof}
This can be obtained from the results in \cite{Gon03}, and can also be found in \cite{BoPa09}.
\end{proof}

\begin{corollary}\label{C:pure_centralizers}
If $\alpha \in P_n$ is a pure braid, $Z_{B_n}(\alpha)=Z_{B_n}(\alpha^m)$ for every $m\neq 0$.
\end{corollary}

\begin{proof}
If $\gamma\in Z_{B_n}(\alpha^m)$ then $\gamma \alpha^m \gamma^{-1} = \alpha^{m}$, so $\gamma \alpha \gamma^{-1}$ and $\alpha$ are two pure $m$-th roots of the same pure braid, hence they are equal, so $\gamma\in Z_{B_n}(\alpha)$.
\end{proof}

\begin{corollary}\label{C:common_power_cyclic}
If $x\in P_n$ and $y\in B_n$ are such that $x^k=y^m$ for some nonzero integers $k$ and $m$, then the subgroup $\langle x,y\rangle\subset B_n$ is cyclic.
\end{corollary}

\begin{proof}
Let $z = x^k = y^m$. Notice that $y\in Z_{B_n}(z) = Z_{B_n}(x^k) = Z_{B_n}(x)$ as $x$ is pure, hence $x$ and $y$ commute. As $B_n$ is torsion-free, then $\langle x,y\rangle$ is isomorphic to either $\mathbb Z^2$ or $\mathbb Z$. But since $x^k=y^m$, it must be isomorphic to $\mathbb Z$.
\end{proof}

\begin{proposition}\label{P:maximal_cyclic}
Every cyclic subgroup $\langle x\rangle \subset P_n$ is contained in a unique maximal cyclic subgroup.
\end{proposition}

\begin{proof}
Suppose first that $x$ is periodic. The set of periodic pure braids is precisely the cyclic group $\langle \Delta^2\rangle$. Therefore, $\langle \Delta^2\rangle$ is the desired maximal cyclic subgroup. Now suppose that $x$ is pseudo-Anosov, with dilatation factor $\lambda$. A $k$-th root of $x$ will then be pseudo-Anosov, with dilatation factor $\sqrt[k]{\lambda}$. As the set of possible dilatation factors is discrete~\cite{AY81}, the set of positive integers $k$ for which $x$ admits a $k$-th root in $P_n$ is finite. The $k$-th root of $x$ for which $k$ is maximal is then a generator for a maximal cyclic subgroup containing $\langle x\rangle$. Finally, suppose that $x$ is reducible non-periodic. Then $x_{ext}$ is either periodic or pseudo-Anosov, and it is also a pure braid (on fewer strands). If $\alpha$ is a pure $k$-th root of $x$, then $\alpha_{ext}$ is a pure $k$-th root of $x_{ext}$ (the canonical reduction system is preserved by taking powers, so $CRS(\alpha)=CRS(x)$, and the external braids are just obtained by emptying the tubes). Hence, as there is a maximal $k$ such that a $k$-th root of $x_{ext}$ exists, there is also a maximal $k$ for which a $k$-th root of $x$ exists. Such $k$-th root is a generator of the desired cyclic subgroup containing $\langle x\rangle$.

The uniqueness of the maximal cyclic subgroup follows from Corollary~\ref{C:common_power_cyclic}.
\end{proof}

\begin{proposition}\label{P:Centralizer_in_Pn}
Let $\alpha\in P_n$ with $n\geq 2$. The centralizer $Z_{P_n}(\alpha)$ is as follows:
\begin{itemize}

  \item If $\alpha$ is periodic, $Z_{P_n}(\alpha)=Z_{P_n}(\langle \Delta^2\rangle)=P_n$.

  \item If $\alpha$ is pseudo-Anosov, $Z_{P_n}(\alpha)\simeq \mathbb Z^2$, consisting of all pure braids admitting a representative which preserves the same two foliations as $\alpha$.

  \item If $\alpha$ is reducible, $Z_{P_n}(\alpha)\simeq Z_{P_{m_1}}(\alpha_1)\times \cdots \times Z_{P_{m_k}}(\alpha_k) \times Z_{P_m}(\alpha_{ext})$, where $\alpha_1,\ldots,\alpha_k$ and $\alpha_{ext}$ are the internal and external braids, respectively, defined for $\alpha$. Moreover, $m_1+\cdots+m_k+m= n+k$.
\end{itemize}
Therefore, in any case, $Z_{P_n}(\alpha)\simeq P_{t_1}\times \cdots \times P_{t_r}\times \mathbb Z^t$ for some $r\geq 0$, where $t_1+\cdots +t_r+t\leq n+r-1$.
\end{proposition}

\begin{proof}
The periodic case is clear, as the only periodic pure braids are the powers of $\Delta^2$. If $\alpha$ is pseudo-Anosov, it follows from Theorem~\ref{T:centralizer_Bn} that $Z_{P_n}(\alpha)$ is the subgroup of $Z_{B_n}(\alpha)\simeq \mathbb Z^2$ consisting of pure braids admitting a representative which preserves the same two foliations as $\alpha$. Now $Z_{P_n}(\alpha)$ cannot be cyclic, as it contains a pseudo-Anosov braid ($\alpha$) and a periodic braid ($\Delta^2$), which cannot have a common power. Hence $Z_{P_n}(\alpha)\simeq \mathbb Z^2$ in this case.

Suppose finally that $\alpha$ is reducible.  By Theorem~\ref{T:centralizer_Bn}, we have that
$$
  Z_{B_n}(\alpha)=\left(Z_{B_{m_1}}(\alpha_1)\times \cdots \times Z_{B_{m_k}}(\alpha_k)\right)\rtimes Z_0(\alpha_{ext}).
$$
Notice that $\alpha_1,\ldots,\alpha_k$ and $\alpha_{ext}$ are pure braids, and that $m_i$ is the number of strands of the braid $\alpha_i$. Recall also that $Z_0(\alpha_{ext})\subset Z_{B_m}(\alpha_{ext})$, where $m$ is the number of strands of $\alpha_{ext}$.

Now $Z_{P_n}(\alpha)= Z_{B_n}(\alpha)\cap P_n$, hence we need to describe the pure elements in the above semidirect product. The left factor just becomes $Z_{P_{m_1}}(\alpha_1)\times \cdots \times Z_{P_{m_k}}(\alpha_k)$. For the second factor, we need to describe the subgroup $Z_0(\alpha_{ext})$ with more detail, as it is done in \cite{GoWi04}. The group $Z_{B_m}(\alpha_{ext})$ (whose elements can be seen in $B_n$ as braids made of strands and tubes, with trivial braids inside the tubes) acts on $Z_{B_{m_1}}(\alpha_1)\times \cdots \times Z_{B_{m_k}}(\alpha_k)$ by permuting coordinates (conjugating in $B_n$ by the tubular braid, just permutes the tubes). Then $Z_0(\alpha_{ext})$ is the subgroup formed by the elements whose induced permutation acts as follows: if $i$ is sent to $j$, then $\alpha_i$ is conjugate to $\alpha_j$.

Now take a pure element in $Z_{B_{m}}(\alpha_{ext})$. Its induced permutation is trivial, hence it belongs to $Z_0(\alpha_{ext})$. This means that $Z_0(\alpha_{ext})\cap P_m = Z_{B_{m}}(\alpha_{ext})\cap P_m = Z_{P_{m}}(\alpha_{ext})$. Moreover, the action of this group on $Z_{P_{m_1}}(\alpha_1)\times \cdots \times Z_{P_{m_k}}(\alpha_k)$ is trivial. Hence the semidirect product becomes a direct product when restricted to $P_n$, and hence:
$$
  Z_{P_n}(\alpha)=Z_{P_{m_1}}(\alpha_1)\times \cdots \times Z_{P_{m_k}}(\alpha_k)\times Z_{P_m}(\alpha_{ext}).
$$

Now recall that $m_1,\ldots,m_k$ and $m$ are, respectively, the number of strands of $\alpha_1,\ldots,\alpha_k$ and $\alpha_{ext}$. By the decomposition of $\alpha$ into the internal braids $\alpha_1,\ldots,\alpha_k$ and the external braid $\alpha_{ext}$, it follows that the punctured disc $\mathbb D_n$ is obtained from the punctured disc $\mathbb D_m$ by replacing $k$ punctures with the $k$ punctured discs $\mathbb D_{m_1},\ldots, \mathbb D_{m_k}$. At each replacement, a puncture is replaced by $m_i$ punctures. Hence $n=m+(m_1-1)+\cdots +(m_k-1)$, so $m_1+\cdots + m_k + m= n+k$ as we wanted to show. We remark that if $\alpha$ is not pure, this equality becomes an inequality, $m_1+\cdots + m_k+ m\leq n+k$, as the permutation induced by $\alpha$ on the tubes may contain nontrivial cycles, and there is just one internal braid for each cycle. But in our case, as $\alpha$ is pure, the equality holds.

Notice that, if we define $r^*=r-1$ for every positive integer $r$, the above equality reads $m_1^*+\cdots+m_k^*+m^*=n^*$.

Now if some $\alpha_i$ is reducible an not periodic, we can decompose $Z_{P_{m_i}}(\alpha_i)$ into a direct product of pure centralizers. We can keep doing this until all braids involved are either periodic or pseudo-Anosov (recall that at every iteration the braids involved have fewer strands, and in $B_2$ all braids are periodic). Hence we will obtain
$$
    Z_{P_n}(\alpha)=Z_{P_{t_1}}(\beta_1)\times \cdots \times Z_{P_{t_s}}(\beta_s),
$$
where we can assume that $\beta_1,\ldots,\beta_r$ are periodic, and $\beta_{r+1},\ldots,\beta_s$ are pseudo-Anosov. Also, by induction on the number of strands it follows that $t_1^*+\cdots+t_s^*=n^*$, that is, $t_1+\cdots +t_s=n+s-1$. But then
$$
    Z_{P_n}(\alpha)=P_{t_1}\times \cdots \times P_{t_r} \times \mathbb Z^{2(s-r)}.
$$
Let $t=2(s-r)$. As every pseudo-Anosov braid has at least 3 strands, one has $3(s-r)\leq t_{r+1}+\cdots+t_s$, so $t_1+\cdots +t_r +t = t_1+\cdots +t_r+ 3(s-r)-(s-r) \leq t_1+\cdots+ t_s -(s-r) = n+s-1 -(s-r)= n+r-1$.
\end{proof}

Concerning the geometric dimension of $B_n$, the case of proper actions is well known for mapping class groups in general: given a mapping class group $\Gamma^b_{g,n}$ of a connected orientable surface of genus $g$, with $n$ punctures and $b$ boundary components, it is known that $\gdp \Gamma^b_{g,n}=\text{vcd}\ \Gamma^b_{g,n}$, where $\text{vcd}\: G$ is the virtual cohomological dimension of a group $G$. This follows from Harer \cite{Har86} for $b>0$, has been showed by Hensel, Osajda, and Przytycki \cite{HOP16} when $b=0$ and $n>0$, and by Aramayona and Mart\'{\i}nez \cite{ArMa16} when $b=0$ and $n=0$.

The virtual cohomological dimension of $\Gamma^b_{g,n}$ was computed by Harer in the mentioned paper. The case we are interested in is $B_n = \Gamma^1_{0,n}$, and by Harer's formula we obtain:
$$
 \gdp B_n= \text{vcd}\ B_n = n-1.
$$

This equality is also consequence of the results in (\cite{Arn70}, section 3).

In this paper we are interested in $\gd B_n$. We can easily find a lower bound thanks to the following result.

\begin{proposition}\label{P:abelian_inside}
The pure braid group $P_n$ (and hence $B_n$) contains a subgroup isomorphic to $\mathbb Z^{n-1}$.
\end{proposition}

\begin{proof}
We define explicit generators, $a_1,\ldots,a_{n-1}$, as follows:
$$
  a_i=(\sigma_i\sigma_{i+1}\cdots \sigma_{n-1})(\sigma_{n-1}\cdots \sigma_{i+1}\sigma_i).
$$
One can check that these elements commute with each other just by drawing the corresponding braids, or by applying the braid relations. As $P_n$ is torsion-free, these elements generate a free abelian subgroup of $P_n$. To check that its rank is $n-1$, just define a projection $\pi: P_n \rightarrow \mathbb Z^{n-1}$ which sends a pure braid $\alpha$ to $\pi(\alpha)=(x_1,\ldots,x_{n-1})$, where $x_i$ is the algebraic number of crossings of the strands $i$ and $n$ in the braid $\alpha$, divided by 2. This projection is a well defined homomorphism, and its restriction to $\langle a_1,\ldots,a_{n-1}\rangle$ is surjective as $\pi(a_i)=(0,\ldots,0,1,0,\ldots,0)$ with the 1 in position $i$.
\end{proof}

\begin{corollary}\label{C:lower_bound}
$\gd B_2 = \gd P_2 = 0$, and $\gd B_n\geq \gd P_n \geq n$ for $n\geq 3$.
\end{corollary}

\begin{proof}
First, $B_2\simeq P_2 \simeq \mathbb Z$, so $\gd B_2 = \gd P_2 = 0$. Assume $n\geq 3$. The first inequality holds as $P_n$ is a subgroup of $B_n$.  Finally, as $\gd \mathbb Z^{n-1}=n$ for $n\geq 3$ (Example 5.21 in \cite{LuWe12}), the second inequality follows from Proposition~\ref{P:abelian_inside}.
\end{proof}

Some upper bounds for $\gd B_n$ are already known. In \cite{JuTr16} such a bound is given for any mapping class group $\Gamma^b_{g,n}$ of a compact orientable surface with punctures and negative Euler characteristic. In the particular case of braid groups with $n\geq 2$ (in order to have negative Euler characteristic) and $m\geq 3$ this bound is:
$$
    \gd B_n \leq [B_n: B_n[m]] \cdot (\text{vcd}\: B_n +1) = [B_n: B_n[m]]\cdot n,
$$
where $B_n[m]$ is the {\em $m$-th congruence subgroup} of $B_n$, that is the subgroup of elements which act trivially on $H_1(\mathbb D_n, \mathbb Z/m)$.

Notice that the fundamental group of $\mathbb D_n$ is the free group $F_n$ of rank $n$, hence the first homology group is $H_1(\mathbb D_n, \mathbb Z/m)\simeq (\mathbb Z/m)^n$. It is well known~\cite{Bir74} that the action of the braid group (by automorphisms) on the fundamental group of $\mathbb D_n$ coincides with the Artin representation $B_n\to \mbox{Aut}(F_n)$, (see~\cite{Art47}). The action induced by this homomorphism on the abelianization of $F_n$ sends a braid $\alpha$ to the automorphism of $\mathbb Z^n$ which just permutes the coordinates in the same way as $\alpha$ permutes its strands. Therefore, the action of $\alpha \in B_n$ on $H_1(\mathbb D_n, \mathbb Z/m)$ is just given by permuting the coordinates in the same way as $\alpha$ permutes its strands. This means that the elements which induce the trivial action are precisely the pure braids. Hence $B_n[m]=P_n$ for every $m\geq 2$.

It follows that $[B_n: B_n[m]]= [B_n: P_n] = n!$ Hence, the above upper bound for $\gd B_n$ is $n!\: n$.  The goal of this paper is to show that $\gd B_n=n$.

\section{Commensurators}
\label{Commensurators}
The main tool that we are going to use to compute geometric dimensions with respect to the family of virtually cyclic groups is L\"{u}ck-Weiermann's Theorem \ref{maintheorem}, and a necessary step in our calculation will be to describe the commensurators of the virtually cyclic subgroups generated by periodic or pseudo-Anosov braids. This task will be undertaken in this section, while in the last one we will use different tools to compute the dimensions in the reducible non-periodic case. We start by reviewing some facts concerning commensurators that will be needed in the rest of the paper.

\begin{definition}
\label{Commdefo}
Given an inclusion $H<G$ of groups, the \emph{commensurator} of $H$ in $G$ is the group:
$$
\emph{Comm}_G(H)=\{x\in G|\ [H:H\cap H^x]<\infty \emph{ and }[H^x:H\cap H^x]<\infty\}.
$$

\end{definition}

We will be mainly interested in the case in which $H$ is a virtually cyclic subgroup of $G$. In this case, the previous definition is equivalent to:
$$\textrm{Comm}_G(H)=\{x\in G |\ |H\cap H^x|=\infty\}.$$

Note that this definition implies that, given the equivalence relation defined just before Definition~\ref{eqclass} over the virtually cyclic infinite subgroups of a group, and $[H]$ an equivalence class, $N_G[H]=\CommH$.

When computing the geometric dimension of $B_n$, we will make frequent use of the following result of Degrijse-Petrosyan:






\begin{theorem}\label{T:gdFH=gdpCommH/K} {\rm (\cite{DePe14}, 4.2)} If $H<G$ is virtually cyclic, and such that $H\simeq\mathbb{Z}$ and $H\unlhd \emph{Comm}_G(H)$, then $\text{\rm gd}_{{\cal{F}}[H]}(\emph{Comm}_G(H))=\gdp(\emph{Comm}_G(H) /H)$.
\end{theorem}


After these preliminaries we are ready for computing the commensurators of the virtually cyclic subgroups of $B_n$. Recall that, as $B_n$ is torsion-free, any non-trivial virtually cyclic group of $B_n$ is infinite cyclic. The following lemma will be very useful, as it identifies the commensurators with normalizers of \emph{subgroups}, in the classic sense:

\begin{lemma}
\label{puregenerator}

Let $H=\langle y \rangle <B_n$ be a cyclic subgroup. Then there exists a braid $x$ such that:

\begin{itemize}

\item $\emph{Comm}_{B_n}(H)=N_{B_n}(\langle x\rangle).$

\item The braids $x$ and $y$ belong to the same cyclic group.

\item The braid $x$ is pure, and no proper root of $x$ is pure.

\end{itemize}
\end{lemma}

\begin{proof}
This result is similar to Proposition 4.8 in \cite{JuTr16}, but we provide a proof adapted to braid groups.

Let $m$ be the smallest positive integer such that $z=y^m$ is a pure braid, and let $x$ be a generator of the maximal cyclic subgroup in $P_n$ containing $\langle z\rangle$ (see Proposition~\ref{P:maximal_cyclic}). We will see that $x$ satisfies the properties in the statement. Clearly, it satisfies the third one by construction.

Notice that $x^k = z = y^m$ for some nonzero integers $k$ and $m$. Hence by Corollary~\ref{C:common_power_cyclic} $\langle x,y\rangle$ is cyclic, so the second property also holds.

Let us finally show that $\Comm_{B_n}(\langle y\rangle)=N_{B_n}(\langle x\rangle)$. Let $\gamma \in \Comm_{B_n}(\langle y\rangle)$. This means that $|\langle \gamma y \gamma^{-1} \rangle \cap \langle y\rangle |=\infty$. As every nontrivial braid has a pure nontrivial power, this intersection must contain an infinite number of pure braids, so $|\langle \gamma x \gamma^{-1} \rangle \cap \langle x\rangle |=\infty$, hence $\gamma x^{p} \gamma^{-1} = x^q$ for some nonzero integers $p$ and $q$. We will now show that $p=\pm q$.

If $x$ is periodic (and pure), $x^p$ is a nontrivial power of $\Delta^2$, which is central. Hence $\gamma x^{p} \gamma^{-1} = x^p$, so $x^p=x^q$ and, as $B_n$ is torsion-free, $p=q$ in this case. If $x$ is pseudo-Anosov and its dilatation factor is $\lambda$, the dilatation factor of $x^p$ is $\lambda^{|p|}$, and the dilatation factor of $x^q$ is $\lambda^{|q|}$. As dilatation factors are preserved by conjugations, it follows that $|p|=|q|$, as we wanted to show. (It is worth mentioning that a conjugation may transform $x$ into its inverse. For instance, in $B_3$ we have $\Delta (\sigma_1\sigma_2^{-1})^{3} \Delta^{-1} = (\sigma_1\sigma_2^{-1})^{-3}$.) Now suppose that $x$ is reducible. Then $x_{ext}$ is either periodic or pseudo-Anosov, and $\gamma$ sends $CRS(x)$ to itself. Hence $p_x(\gamma) (x_{ext})^p  p_x(\gamma)^{-1} = (x_{ext})^q$. Applying the claim to $x_{ext}$, it follows that $p=\pm q$, as desired.

If $p=q$ then $\gamma x^q \gamma^{-1} = x^q$, so $\gamma\in Z_{B_n}(x^q)=Z_{B_n}(x)$ and then $\gamma \in Z_{B_n}(\langle x\rangle) \subset N_{B_n}(\langle x \rangle)$.  If $p=-q$ then $\gamma x^{-q} \gamma^{-1} = x^q$, so $\gamma x^{-1} \gamma^{-1}$ and $x$ are two pure $q$-th roots of the same pure braid, hence they coincide, and then $\gamma \in N_{B_n}(\langle x\rangle)$.

Conversely, suppose that $\gamma \in N_{B_n}(\langle x\rangle)$. This means that either $\gamma x \gamma^{-1} = x$ or $\gamma x \gamma^{-1} = x^{-1}$. Recall that $y^k=z=x^m$ for some nonzero integers $k$ and $m$. Then $ \gamma y^k \gamma^{-1} = \gamma x^m \gamma^{-1} = x^{\pm m} = y^{\pm k} \in \langle \gamma y \gamma^{-1}\rangle \cap \langle y \rangle = H^\gamma \cap H$, so $|H^\gamma \cap H | = \infty$. This shows that $\gamma \in \Comm_{B_n}(H)$, finishing the proof.
\end{proof}

We shall use the following results later. They can be deduced from the previous lemma, but the first one can also be showed independently.

\begin{lemma}\label{L:comm_periodic}
If $y\in B_n$ is periodic and nontrivial, then $\emph{Comm}_{B_n}(\langle y\rangle)=B_n$.
\end{lemma}

\begin{proof}
If $y$ is periodic, $y^m$ is central for some nonzero integer $m$, that is, $y^m=\Delta^{2k}$ for some nonzero integer $k$. Then, for every $\alpha\in B_n$, the conjugate $\alpha y \alpha^{-1}$ is such that $(\alpha y \alpha^{-1})^m = \alpha \Delta^{2k} \alpha^{-1} = \Delta^{2k}$. Hence $\langle y \rangle \cap \langle \alpha y \alpha^{-1} \rangle \supset \langle \Delta^{2k}\rangle$, so $\alpha \in \Comm_{B_n}(\langle y\rangle)$.
\end{proof}

\begin{corollary}\label{C:comm_pA}
If $y\in B_n$ is pseudo-Anosov, then $\emph{Comm}_{B_n}(\langle y\rangle)$ is isomorphic to either $\mathbb Z^2$ or an extension of $\mathbb Z^2$ by $\mathbb Z/2$.
\end{corollary}

\begin{proof}
By Lemma~\ref{puregenerator}, there is some $x\in P_n$ such that $\Comm_{B_n}(\langle y\rangle)=N_{B_n}(\langle x \rangle)$ As $x$ is a root of a power of $y$, it is also pseudo-Anosov, hence $Z_{B_n}(\langle x \rangle)\simeq \mathbb Z^2$. Now notice that for any $\alpha\in N_{B_n}(\langle x \rangle)$, either $\alpha^{-1} x \alpha =x$ or $\alpha^{-1} x \alpha = x^{-1}$. If all elements are of the first kind, then $N_{B_n}(\langle x \rangle)= Z_{B_n}(\langle x \rangle)\simeq \mathbb Z^2$. Otherwise, we have the following short exact sequence:
$$
    1\longrightarrow Z_{B_n}(\langle x \rangle) \longrightarrow N_{B_n}(\langle x \rangle) \stackrel{p}{\longrightarrow} \mathbb Z/2 \longrightarrow 1
$$
where $p(\alpha)=0$ if $\alpha^{-1} x \alpha =x$, and $p(\alpha)=1$ if $\alpha^{-1} x \alpha = x^{-1}$. As $Z_{B_n}(\langle x \rangle)\simeq \mathbb Z^2$ the result holds.
\end{proof}

\section{Classifying space for the family of virtually cyclic groups}

In this section we will compute $\gd B_n$, which is the major achievement of this paper. Our result is:

\begin{theorem}
\label{teoremazo}

The geometric dimension of $B_n$ with respect to the family of virtually cyclic subgroups is $n$, for $n\geq 3$.

\end{theorem}

We devote this whole section to prove this result, and assume from now on that $n\geq 3$. According to Corollary \ref{C:lower_bound}, $\gd B_n\geq n$. Then, we must prove that $n$ is also an upper bound for this geometric dimension. As said above, we base our arguments on  the main theorem \ref{maintheorem} of L\"{u}ck-Weiermann, and we need to bound the three geometric dimensions that appear Corollary \ref{C:cotas_suficientes}. The first two bounds are quite straightforward, and we will deal with them later. So let us concentrate in  $\textrm{gd}_{\F[H]}N_{B_n}[H]$, which has to be upper bounded by $n$.

Recall from Lemma~\ref{puregenerator} that we can assume that $H$ is generated by a pure braid $x$ which admits no proper roots, and that $N_{B_n}[H]=\Comm_{B_n}(H)=N_{B_n}(H)$.  But every subgroup is normal in its normalizer, hence by Theorem~\ref{T:gdFH=gdpCommH/K} we have $\text{\rm gd}_{{\cal{F}}[H]}(N_{B_n}[H])=\gdp(N_{B_n}[H] /H)$. Hence we need to upper bound $\gdp(N_{B_n}[H] /H)$ by $n$.

We invoke again Nielsen-Thurston theory to divide the computation in three cases. Notice that if the original generator of $y$ was periodic (respectively pseudo-Anosov or reducible), then so is the pure element $x$ given by Lemma \ref{puregenerator}, as it is a root of a power of $y$.

\textbf{1. Periodic case}. Suppose that $H=\langle x\rangle$ where $x$ is periodic. According to Lemma~\ref{L:comm_periodic}, $\Comm_{B_n}(H)=B_n$, and as $x$ is pure, periodic and has no proper roots, $x=\Delta^{\pm 2}$. Therefore $ \gdp(N_{B_n}[H] /H)=\gdp (B_n/Z(B_n))$, being $Z(B_n)=\langle \Delta^2\rangle$ the center of $B_n$. It is known that $B_n/\langle \Delta^2\rangle$ is a subgroup of $\Gamma_{0,n+1}$, the mapping class group of the sphere with $n+1$ punctures, and it is a consequence of work of \cite{HOP16} and \cite{Har86} that the proper geometric dimension of $\Gamma_{0,n+1}$ is $n-2$. Hence, $\gdp(N_{B_n}[H] /H)\leq n-2$ in this case.

\textbf{2. Pseudo-Anosov case}. Suppose now that $x$ is pseudo-Anosov. By Corollary~\ref{C:comm_pA}, we know that $\Comm_{B_n}(H)$ is either $\mathbb Z^2$ or an index 2 extension of $\mathbb Z^2$. Without loss of generality, we may assume that $x$ is a power of one of the two generators of $\mathbb Z^2$. Hence, $\Comm_{B_n}(H)/H$ is isomorphic to a finite extension of $\mathbb{Z}\oplus\mathbb{Z}/m$ ($m=1$ is allowed). Applying Theorem 5.26 of \cite{Luc05}, we obtain that $\underline{\textrm{gd}}\textrm{ }\Comm_{B_n}(H)/H=1$, the Hirsch number, and hence $\gdp(N_{B_n}[H] /H)= 1$ in this case.


\textbf{3. Reducible, non-periodic case.} Suppose finally that $x$ is reducible and not periodic. We will find an upper bound $\gdp (N_{B_n}(H)/H)$ basing our strategy in the following result of C. Mart\'{i}nez (Theorem 2.5 in \cite{Mar13}):

\begin{theorem}
\label{Conchita}
For any discrete group $\Gamma$, the following inequality holds:

$$\underline{\textrm{\rm cd}}\textrm{ }\Gamma\leq \max_{F\in \mathcal F_{\text{\rm Fin}}(\Gamma)}\{ \text{\rm pd}_{\mathbb{Z}W_\Gamma F}B(W_\Gamma F)+\text{\rm rk}(W_{\Gamma}F)\}.$$

\end{theorem}

Here pd is the projective dimension over the corresponding group ring, $W_{\Gamma}F=N_{\Gamma}F/F$ is the Weyl group, and $B(W_\Gamma F)$ is the module of bounded functions over $W_\Gamma F$ (see the mentioned paper for details about these concepts). Moreover, $\text{rk}(W_{\Gamma}F)$ is the rank of a maximal elementary abelian subgroup of $W_\Gamma F$, as usual. Recall from Proposition \ref{Eilenberg} that for any group $G$, $\underline{\textrm{cd}}\textrm{ }G=\gdp G$ except possibly in the case $\gdp G=3$. From now on we assume $\Gamma=N_{B_n}(H)/H$, and we should bound the previous sum.

We first deal with the term $\text{\rm pd}_{\mathbb{Z}W_\Gamma F}B(W_\Gamma F)$. By Lemma 7.3 in \cite{KrMi98}, $\text{\rm pd}_{\mathbb{Z}W_\Gamma F}B(\mathbb Z W_\Gamma F)\leq \text{\rm pd}_{\mathbb{Z}\Gamma}B(\mathbb Z \Gamma)$ for every finite $F<\Gamma$. Moreover, if we can prove that $\Gamma$ is virtually torsion-free, Lemma 3.9 in \cite{Mar07} implies that the latter coincides with $\text{\rm vcd}(\Gamma)$, the virtual cohomological dimension of $\Gamma$, and we would have:
$$
   \underline{\text{\rm cd}}\Gamma \leq \max_{F\in \mathcal F_{\text{\rm Fin}}(\Gamma)}\{\text{\rm vcd}(\Gamma)+\text{\rm rk}(W_{\Gamma}F)\}.
$$
Let us then show that this is true:

\begin{lemma}
The group $\Gamma$ is virtually torsion-free.
\end{lemma}

\begin{proof}
We will use the short exact sequence $1\rightarrow P_n \rightarrow B_n \stackrel{p}{\rightarrow} \Sigma_n \rightarrow 1$, which we can restrict to $N_{B_n}(H)$:
$$
    1 \longrightarrow N_{B_n}(H)\cap P_n \longrightarrow N_{B_n}(H) \stackrel{p}{\longrightarrow} G \longrightarrow 1,
$$
where $G=p(N_{B_n}(H))\subset \Sigma_n$ is a finite group.

Recall that $H=\langle x \rangle $ where $x$ is pure. Then $H\unlhd (N_{B_n}(H)\cap P_n)$ and $H\unlhd N_{B_n}(H)$, so we can quotient the above injection by $H$ and we obtain:
$$
    1 \longrightarrow (N_{B_n}(H)\cap P_n)/H \longrightarrow \Gamma \stackrel{p}{\longrightarrow} G \longrightarrow 1.
$$
We need to show that $(N_{B_n}(H)\cap P_n)/H$ is torsion-free.

Let $\alpha H \in (N_{B_n}(H)\cap P_n)/H$ be a torsion element. That is, $\alpha\in N_{B_n}(H)\cap P_n$ is such that $\alpha^m \in H$ for some nonzero integer $m$. This means that $\alpha^m=x^k$ for some nonzero integers $m$ and $k$. By Corollary~\ref{C:common_power_cyclic}, $\langle \alpha,x\rangle$ is a cyclic subgroup, which belongs to $P_n$ as it is generated by pure braids. As $\langle x \rangle$ is a maximal cyclic subgroup in $P_n$, we get $\langle \alpha, x \rangle =\langle x \rangle=H$, so $\alpha\in H$ and $\alpha H$ is trivial in $(N_{B_n}(H)\cap P_n)/H$. Hence $(N_{B_n}(H)\cap P_n)/H$ is torsion-free.
\end{proof}

Now we can obtain our bound for the projective dimension.

\begin{lemma}
\label{vcd}
The virtual cohomological dimension of $\Gamma$ is bounded above by $n-2$.
\end{lemma}

\begin{proof}
We saw in the previous lemma that $(N_{B_n}(H)\cap P_n)/H$ is a torsion-free, finite index subgroup of $\Gamma$. Hence, we must show that $\text{cd}\left((N_{B_n}(H)\cap P_n)/H\right)\leq n-2$.

Notice that $N_{B_n}(H)\cap P_n= N_{P_n}(H)$. Now we point out that the pure braid group $P_n$ is bi-orderable \cite{RZ98}. This means that there is a total order of its elements which is invariant under multiplication on the left and also on the right. In such an order, if an element is positive (greater than the neutral element) its inverse is negative, and viceversa. Recall that $H=\langle x \rangle$. If $x>1$, then $ax>a$ and hence $axa^{-1}>1$ for every $a\in P_n$. Analogously, if $x<1$ then $axa^{-1}<1$ for every $a\in P_n$. This means that, in $P_n$, $x$ cannot be conjugated to its inverse. Therefore $N_{P_n}(H)=Z_{P_n}(H)$.

Now recall from Proposition~\ref{P:Centralizer_in_Pn} that $Z_{P_n}(H)=P_{t_1}\times \cdots \times P_{t_r}\times \mathbb Z^t$, where $r,t\geq 0$ and $t_1+\cdots +t_r+t\leq n+r-1$.

Let $C$ be the center of $Z_{P_n}(H)$. As the center of each $P_{t_i}$ is cyclic (generated by $\Delta_{t_i}^2$), it follows that $C\simeq \mathbb Z^{r+t}$. Denote by $G_i$ the quotient of $P_{t_i}$ by its center, which is a (finite index) subgroup of the mapping class group $\Gamma_{0,t_i+1}$ of the $(t_i+1)$-times punctured sphere. Then $Z_{P_n}(H)/C= G_1\times \cdots \times G_r$.

We then have the following short exact sequence, which is a central extension:
$$
    1 \longrightarrow C \longrightarrow Z_{P_n}(H) \longrightarrow G_1\times \cdots \times G_r \longrightarrow 1.
$$
Moreover, as $H$ is central in $Z_{P_n}(H)$, we can quotient by $H$ and obtain:
$$
    1 \longrightarrow C/H \longrightarrow Z_{P_n}(H)/H \longrightarrow G_1\times \cdots \times G_r \longrightarrow 1.
$$
Recall that we want to bound the cohomological dimension of the middle group. It is important to remark that all groups involved in the above exact sequence are torsion-free: First, we have shown in the previous lemma that $Z_{P_n}(H)/H$ is torsion-free, hence so is its subgroup $C/H$. On the other hand, a torsion element in $G_i=P_{t_i}/\langle \Delta_{t_i}^2\rangle$ comes from a periodic element in $P_{t_i}$. But the only pure braids which are periodic in $P_{t_i}$ are the powers of $\Delta_{t_i}^2$, hence $G_i$ is torsion-free for every $i$.





Now it follows from Theorem 5.15 in \cite{Luc05} that $\text{\rm gd}\ G=\gdp G$ is subadditive for extensions of torsion-free groups. Hence
$$
  \text{\rm gd}(Z_{P_n}(H)/H) \leq \text{\rm gd}(C/H) + \text{\rm gd}(G_1) +\cdots + \text{\rm gd}(G_r).
$$
Now $C\simeq \mathbb Z^{r+t}$, and $H$ is a cyclic subgroup which is maximal. Hence $C/H\simeq \mathbb Z^{r+t-1}$, and $\text{gd}(C/H) = r+t-1$. Also, each $G_i$ is a subgroup of the mapping class group $\Gamma_{0,t_i+1}$, so $\text{gd}(G_i) = \gdp(G_i)\leq \gdp(\Gamma_{0,t_i+1})=t_i-2$.  We then have:
$$
 \text{gd}(Z_{P_n}(H)/H) \leq (r+t-1)+t_1+\cdots +t_r -2r \leq  r-1+(n+r-1)-2r= n-2.
$$
Therefore $\text{\rm gd}(Z_{P_n}(H)/H) \leq  n-2$. As $\text{\rm cd}\ G = \text{\rm gd}\ G$ except in the case $\text{\rm cd}\ G=2$ and $\text{\rm gd}\ G=3$, it follows that $\text{cd}(Z_{P_n}(H)/H)\leq n-2$.
\end{proof}

\begin{remark}

A different proof of the previous lemma follows from a direct application of the Gysin sequence (see for example \cite{Whi78}, 5.12), but we prefer the one we provide because it includes an explicit description of the normalizer.

\end{remark}

Now, in order to bound the second term of the sum in Theorem \ref{Conchita}, we need to understand the torsion of $\Gamma$.

\begin{lemma}\label{L:finite_is_cyclic}
Any finite subgroup of $\Gamma$ is cyclic.
\end{lemma}

\begin{proof}
Consider the natural projection $p: N_{B_n}(H) \rightarrow N_{B_n}(H)/H = \Gamma$. Let $F$ be a finite subgroup of $\Gamma$, and let $G=p^{-1}(F)$. Consider the restriction of $p$ to $G$. Its image is finite, and its kernel is a subgroup of $H$, so is cyclic. Hence $G$ is a virtually cyclic group. As $G\subset B_n$, which is torsion-free, $G$ is cyclic. Hence $p(G)=F$ is also cyclic.
\end{proof}

We are now in a good position to deal with the rank of Weyl groups, and our bound will be a consequence of the following lemma:

\begin{lemma}
For every finite $F<\Gamma$, every finite group of the Weyl group $W_\Gamma F$ is cyclic, and in particular, $rk(W_\Gamma F)=1$.
\label{Weylrank}
\end{lemma}

\begin{proof}
Let $F<\Gamma$ be finite. Consider the short exact sequence:
$$
   1 \longrightarrow F \longrightarrow N_\Gamma F \stackrel{p}{\longrightarrow} N_\Gamma F/F \longrightarrow 1,
$$
where $p$ is the natural projection.

Let $X$ be a finite subgroup of $N_\Gamma F/F=W_\Gamma F$, and let $Y=p^{-1}(X)$. Then $Y$ is an extension of $Y\cap F$ by $X$, which are both finite groups. Hence $Y$ is finite. As $Y\subset N_\Gamma F \subset \Gamma$, by Lemma~\ref{L:finite_is_cyclic} $Y$ is cyclic. Hence $X=p(Y)$ is also cyclic.

Observe, in particular, that this implies that $rk(W_\Gamma F)=1$ for every finite $F$ in $\Gamma$.
\end{proof}

Now we can join the previous information to obtain:

\begin{proposition}

The proper geometric dimension of $\Gamma$ is smaller or equal to $n-1$, except possibly in the case when $\gdp\Gamma=3$, in which it can be at most $n$.

\end{proposition}

\begin{proof}

By Theorem \ref{Conchita}, Lemma \ref{vcd} and Lemma \ref{Weylrank}, we have $\underline{\rm{cd}}\ \Gamma\leq n-1$.
Hence by Proposition \ref{Eilenberg}, $\underline{\rm{cd}}\ \Gamma = \gdp \Gamma \leq n-1$, except possibly when $\underline{\text{cd}}\ \Gamma =2$ and $\gdp \Gamma =3$.
\end{proof}


\emph{Proof of Theorem \ref{teoremazo}}.

Let $B_n$ be the full braid group, for $n\geq 3.$ First, as $P_n<B_n$, Corollary \ref{C:lower_bound} implies that $\gd B_n\geq n$.

To check the other equality, we make use of L\"{u}ck-Weiermann model, and in particular the dimensional conditions of Corollary \ref{C:cotas_suficientes}:

 \begin{itemize}

 \item We know (by \cite{Arn70} or \cite{Har86}, see Section~\ref{S:braids} above) that the geometric dimension $\underline{\textrm{gd}}\textrm{ } B_n$ is always $n-1$, so it is smaller than $n$.

 \item As $N_{B_n}[H]\subseteq B_n$ for every cyclic $H< B_n$, $\underline{\textrm{gd}}\textrm{ } N_{B_n}[H]\leq \underline{\textrm{gd}}\textrm{ } B_n\leq n-1$.

 \item Let $n\geq 3$ and $H<B_n$ cyclic. If $H$ is generated by a periodic or a pseudo-Anosov braid, it was stated above that $\textrm{gd}_{{{\cal{F}}}[H]}N_{B_n}[H]\leq n-2$, and in particular smaller than $n$. If $H$ is generated by a reducible non-periodic braid,
     the previous proposition states that $\textrm{gd}_{{{\cal{F}}}[H]}N_{B_n}[H]\leq n$.


 \end{itemize}
Now appealing to Corollary \ref{C:cotas_suficientes}, $\gd B_n\leq n$, and thus $\gd B_n=n$. So we are done. \quad\qed\medskip

The same is true for pure subgroups of $B_n$:

\begin{corollary}

We have $\gde P_n=n$ for $n\geq 3$.

\end{corollary}

\begin{proof}
The bound $\gd P_n\geq n$ was proved in Corollary \ref{C:lower_bound}, while $\gd P_n\leq n$ is a consequence of Theorem \ref{teoremazo}, taking into account Remark \ref{subfamilies}.
\end{proof}


\vspace{.3cm}
\noindent {\footnotesize
\begin{minipage}[t]{6cm}
{\bf Ram\'on Flores:} \\
Dpto. Geometr\'{\i}a y Topolog\'{\i}a. \\
Facultad de Matem\'aticas\\
Instituto de Matem\'aticas (IMUS)\\
Universidad de Sevilla \\
Av. Reina Mercedes s/n\\
41012 Sevilla (SPAIN)\\
E-mail: ramonjflores@us.es
\\ URL: www.personal.us.es/ramonjflores

\end{minipage}
\hfill
\begin{minipage}[t]{5.4cm}
{\bf Juan Gonz\'alez-Meneses:} \\
Dpto. \'{A}lgebra.\\
Facultad de Matem\'{a}ticas\\
Instituto de Matem\'aticas (IMUS)\\
Universidad de Sevilla. \\
Av. Reina Mercedes s/n \\
41012 Sevilla (SPAIN)
\\ E-mail:  meneses@us.es
\\ URL: www.personal.us.es/meneses
\end{minipage}
}

\end{document}